\documentclass[smallextended,envcountsect]{svjour3}
\usepackage{latexsym,amsmath,amsfonts,amscd,amssymb,verbatim,array,mathdots,mathtools}
\usepackage{algorithm}
\usepackage{graphicx}
\usepackage[normalem]{ulem}
\smartqed

\usepackage{hyperref}
\usepackage{cleveref}
\usepackage{chngcntr}

\usepackage{color}

\newcommand{\hieu}[1]{\textcolor{black}{{#1}}}

\counterwithin{equation}{section}

\DeclareMathOperator{\Q}{\mathbb{Q}}
\DeclareMathOperator{\R}{\mathbb{R}}
\DeclareMathOperator{\N}{\mathbb{N}}
\DeclareMathOperator{\C}{\mathbb{C}}

\DeclareMathOperator{\I}{\mathcal{I}}
\DeclareMathOperator{\D}{\mathcal{D}}

\DeclareMathOperator{\LT}{LT}
\DeclareMathOperator{\lex}{lex}

\DeclareMathOperator{\rad}{rad}
\DeclareMathOperator{\Sol}{Sol}
\DeclareMathOperator{\vecx}{vec}
\DeclareMathOperator{\lnn}{ln}
\DeclareMathOperator{\Solln}{Sol_{\lnn}}
\DeclareMathOperator{\spp}{sp}
\DeclareMathOperator{\Solsp}{Sol_{\spp}}

\DeclareMathOperator{\CP}{CP}

\DeclareMathOperator{\shape}{\mathtt{shape}}

\begin{document}
\setlength{\baselineskip}{12truept}
\title{Univariate Representations of Solutions to Generic Polynomial Complementarity Problems }
\titlerunning{Univariate Representations of Solutions to Generic Polynomial CPs}
\author{Vu Trung Hieu \and Alfredo Noel Iusem \and Paul Hugo Schm\"{o}lling  \and Akiko Takeda}

\institute{
 Vu Trung Hieu (trunghieu.vu@riken.jp, hieuvut@gmail.com)\\
Center for Advanced Intelligence Project, RIKEN\\
103-0027, Tokyo, Japan
\medskip  \\
 Alfredo Noel Iusem (alfredo.iusem@fgv.br)
\\ 
School of Applied Mathematics, Fundação Getúlio Vargas\\
 Rio de Janeiro, RJ 22250-900, Brazil
\medskip \\
Paul Hugo Schm{\"o}lling (paul.schmolling@ntnu.no)\\
Department of Mathematical Sciences, Norwegian University of Science and Technology 
\\ 7491, Trondheim, Norway
\medskip \\
Akiko Takeda (takeda@mist.i.u-tokyo.ac.jp, akiko.takeda@riken.jp)\\
Graduate School of Information Science and Technology\\
The University of Tokyo\\
113-8656, Tokyo, Japan
 \smallskip \\ 
Center for Advanced Intelligence Project, RIKEN \\
103-0027, Tokyo, Japan
}

\date{
}

\maketitle

\medskip

\begin{abstract}
\hieu{By using the squared slack variables technique, we demonstrate that the solution set of a general polynomial complementarity problem is the image, under a specific projection, of the set of real zeroes of a system of polynomials.}
This paper points out that, generically, this polynomial system has finitely many complex zeroes. In such a case, we use symbolic computation techniques to compute a univariate representation of the solution set. Consequently, univariate representations of special solutions, such as least-norm and sparse solutions, are obtained. After that, enumerating solutions boils down to solving problems governed by univariate polynomials. We also provide some experiments on small-scale problems with worst-case scenarios. At the end of the paper, we propose a method for computing approximate solutions to copositive polynomial complementarity problems that may have infinitely many solutions. 
\end{abstract}

\keywords{polynomial complementarity problem \and solution set \and least-norm solution \and sparse solution \and  zero-dimensional ideal \and univariate representation \and Gr\"{o}bner basis \and  Shape Lemma}

\subclass{90C33 \and 13P10 \and 13P15 \and 65H04}

\section{Introduction}
Let $f:\R^n \to \R^n$ be a mapping. The complementarity problem, named as in \cite{facchinei2003finite}, given for  $f$ is the problem of finding solutions to the following system:
\[x \geq 0, \ f (x) \geq 0, \  \left\langle x \mid f(x) \right\rangle =0,\]
where $\langle \alpha \mid \beta\rangle$ is the usual scalar product of two vectors $\alpha, \beta$ in the Euclidean space $\R^n$.
If $f$ is an affine mapping, the problem becomes a linear complementarity problem, first introduced by Cottle and Dantzig \cite{cottle1968complementary}. If $f$ is a real polynomial mapping, i.e. all the components of $f$ are polynomial, then the problem is called a polynomial complementarity problem (PCP, for short) \cite{gowda2016polynomial}. Moreover, a PCP is a special case of the polynomial variational inequality \cite{hieu2020solution}.

Qualitative properties of PCPs and extension versions have recently been studied; such as existence, stability, and compactness of the solution sets \cite{gowda2016polynomial,hieu2020solution,zheng2020nonemptiness,hieu2020notes,ling2019properties,wang2020existence,pham2021complementary}, bounds of the solution sets \cite{ling2018error,xu2024bounds,li2024lower} or strict feasibility \cite{li2024strict}. The method using semidefinite relaxations in \cite{zhao2019semidefinite} can be modified to enumerate numerically the solutions of a general PCP with finitely many solutions. Computing solutions to a general PCP is a difficult problem, even such a task for problems governed by affine mappings is an NP-hard problem \cite{chung1989np}.

In the case where $f$ is defined by a tensor, the problem is called a tensor complementarity problem \cite{che2016positive,song2016tensor}. Multi-player non-cooperative games, hyper-graph clustering problems, and traffic equilibrium problems can be modeled as tensor complementarity problems \cite{huang2019tensor}.
In the literature, there are several works on solving such tensor problems \cite{qi2019tensor} by using methods based on systems of tensor equations \cite{luo2017sparsest,xu2019equivalent,xie2017iterative}, potential reduction \cite{zhang2019potential}, homotopy continuation \cite{han2019continuation},  smoothing reformulations \cite{huang2017formulating}, the modulus-based reformulation \cite{liu2018tensor}, reformulated programming problems \cite{luo2017sparsest,xie2017iterative}, reformulated integer programming \cite{du2019mixed}, and semidefinite relaxations \cite{zhao2019semidefinite}.

This work is motivated by open problems raised by Huang and Qi in \cite[Problem~1]{huang2019tensor} and \cite[Problem~4]{qi2019tensor} related to using techniques from semi-algebraic geometry to study tensor complementarity problems and the number of solutions in the case where the solution set has finitely many points. Our proposed method addresses these issues for a broad class of problems that includes tensor problems. We are also interested in algorithms computing least-norm and sparse solutions of PCPs; It should be recalled that several algorithms have been proposed for linear or tensor problems, but with special structures \cite{chen2016sparse,xie2017iterative,luo2017sparsest}.

In this paper, we employ the squared slack variables technique, which has been used by many researchers, see e.g. \cite{fukuda2017note} and references therein, to demonstrate that the solution set of a PCP is the image of a real algebraic set under a certain projection, see \Cref{prop:Sol_slack}. This result indicates that the problem is solvable if and only if the real algebraic set is nonempty.

We introduce a class of PCPs given by polynomial mappings $f$ such that the polynomial systems $\{x_1f_1,\dots,x_nf_n\}$ have finitely many complex zeros. A PCP belonging to this class is called $\D_0$, and it should be highlighted that, as shown in \Cref{prop:genericity}, the property of being $\D_0$ is generic.
Using Gr\"{o}bner bases, the Shape Lemma, and results from computational commutative algebra, we prove in \Cref{thm:solCP} the existence of a univariate representation for the solution set of a $\D_0$ problem. Based on the proof, we design an algorithm to compute such a representation, see \Cref{alg:Sol}. Additionally, algorithms are proposed to compute univariate representations of least-norm or sparse solutions. Once a univariate representation is obtained, we can enumerate the solutions by using numerical computation. Furthermore, we discuss the sharp upper bound on the number of solutions of a $\D_0$ problem based on the well-known Bézout bound.

We are also interested in solving PCPs that may not be $\D_0$, whose solution sets are infinite. This paper discusses copositive polynomial mappings and proposes a method to approximate their solutions.

The rest of this paper consists of six sections. Section \ref{sec:pre} gives some definitions, notations, and auxiliary results of polynomial complementarity problems and computational commutative algebra. Section \ref{sec:slack} gives results concerning the squared slack variables technique. Results on the class of $\D_0$ problems are presented in \Cref{sec:CP} and \Cref{sec:compRn}.
\Cref{sec:enum} is about computing numerical solutions for $\D_0$ problems and gives some experiments. The last section, \Cref{sec:copositive}, discusses copositive problems.

\section{Preliminaries}\label{sec:pre}
This section recalls basic notions and
results from complementarity problems and computational commutative algebra. Further details can be found in \cite{facchinei2003finite} and \cite{cox2013}.

\subsection{Polynomial Complementarity Problems}
Let $\R$ denote the field of real numbers. The notation $\langle \alpha \mid \beta\rangle$ is used for the usual scalar product of two vectors $\alpha, \beta$ in the Euclidean space $\R^n$, whereas $\|\alpha\|$ denotes the Euclidean norm of $\alpha$. The index set $\{1,\dots,n\}$ is written as  $[n]$ for brevity.

Let $f=(f_1,\dots,f_n): \R^n \to \R^n$ be a real polynomial mapping.
The \textit{polynomial complementarity problem}, introduced in \cite{gowda2016polynomial}, for $f$ is the problem of finding a solution $\alpha\in \R^n$  of the following system:
\begin{equation}\label{eq:cpc}
     x \geq 0, \ f  (x) \geq 0, \  \left\langle x \mid f(x) \right\rangle =0.
\end{equation}

The problem and its solution set are, respectively, denoted by $\CP(f)$ and $\Sol(f)$. Suppose that $\alpha$ is a solution, so the first two conditions in \eqref{eq:cpc} say that $\alpha$ and $f(\alpha)$ belong concurrently to the non-negative orthant of $\R^n$, and the last condition means that $\alpha$ must belong to the surface given by the following equation:
\begin{equation}\label{eq:sumfx}
    x_1f_1(x) +\dots+ x_nf_n(x) =0.
\end{equation}
\begin{remark}
    It is not difficult to see that $\alpha$ is a solution of $\CP(f)$ if and only if $\alpha$ satisfies the following system:
    \begin{equation}\label{eq:fxzero}
x \geq 0, \ f(x) \geq 0, \  x_1f_1(x) = \dots =  x_nf_n(x) = 0.
    \end{equation}
In this setting, the polynomial equation \eqref{eq:sumfx} can be written equivalently as a square polynomial system, i.e., as $n$ polynomial equations in $n$ variables. This is the key idea allowing us to apply techniques in solving polynomial equations, especially when the square system has finitely many complex roots.
\end{remark}

A \textit{least-norm} solution of $\CP(f)$ is defined as a solution to the following optimization problem:
\begin{equation}\label{eq:minnorm}
    \min \{\|\alpha\|: \alpha \in\Sol(f)\}.
\end{equation}
We denote by $\Solln(f)$ the set of least-norm solutions of $\CP(f)$. Generally, $\Sol(f)$ is not convex, which implies that the cardinality of $\Solln(f)$ may not be one. It should be noted that if we replace the norm function $\|\cdot\|$ in \eqref{eq:minnorm}  by its square $\|\cdot\|^2$, then the set $\Solln(f)$ does not change. 

One defines a \textit{sparse} solution of $\CP(f)$ as a solution to the following optimization problem:
\begin{equation*}\label{eq:norm_0}
    \min \{\|\alpha\|_0: \alpha \in\Sol(f)\},
\end{equation*}
where $\|\alpha\|_0$ is the number of nonzero coordinates of $\alpha$. The set of sparse solutions of $\CP(f)$ is denoted by $\Solsp(f)$.

\subsection{Zero-dimensional Ideals and the Shape Lemma}
We denote by $\C$ the field of complex numbers that is the algebraic closure of the field $\R$. The ring of all polynomials in variables $x=(x_1,\dots,x_n)$
 with real coefficients is denoted by $\R[x]$.
An additive subgroup $\I$ of $\R[x]$  is said
to be an \textit{ideal} of $\R[x]$ if $pq\in \I$ for any $p\in \I$ and
$q\in\R[x]$. Given $p_1,\dots,p_s $ in $ \R[x]$, we denote by
$\left\langle
p_1,\dots,p_s \right\rangle$ the ideal generated by $p_1,\dots,p_s$. If $\I$
is an ideal of $\R[x]$ then, according to Hilbert's basis theorem
(see, for example Theorem~4 in \cite[Chapter 2, \S 5]{cox2013}),
there exist $p_1,\dots,p_s \in \R[x]$ such that
$\I=\left\langle p_1,\dots,p_s \right\rangle$.

Let $\I$ be an ideal of $\R[x]$. The \textit{complex
    algebraic variety} associated to $\I$ is defined as
\[V_{\C}(\I)\coloneqq\{x\in{\C}^n:g(x)=0, \forall g\in \I\}.\] The \textit{real} algebraic
variety associated to $\I$ is $V_{\R}(\I)\coloneqq V_{\C}(\I)\cap\R^n$. Recall that the ideal $\I$
is \textit{zero-dimensional} if the
cardinality of $V_{\C}(\I)$, denoted by $\# V_{\C}(\I)$, is finite. The \textit{radical} of $\I$ is the following set:
\[\rad(\I)\coloneqq\{p\in\R[x]:p^k\in \I \text{ for some }k \in \N\}.\]
The set
$\rad(\I)$ is also an ideal that contains $\I$.
One says that the ideal $\I$ is \textit{radical} if $\I=\rad(\I)$.

The following lemma asserts that the radical of an ideal is radical; Moreover, the complex (real) varieties of $\I$ and of its radical coincide, see, for example, Lemma~5 and Theorem~7  in \cite[Chapter 4, \S 2]{cox2013}:

\begin{lemma}\label{lm:radC}Let $\I$ be an ideal of $\R[x]$. Then, $\rad(\I)$ is radical. Moreover, one has $V_{\C}(\I)=V_{\C}(\rad(\I))$, hence $V_{\R}(\I)=V_{\R}(\rad(\I))$.
\end{lemma}


Let $<$ be a monomial ordering on $\R[x]$
and $\I\neq \{0\}$ be an ideal of $\R[x]$. We denote
by $\LT_<(\I)$ the set of all leading terms $\LT_<(p)$ for $p\in \I$, and  by
$\left\langle \LT_<(\I)\right\rangle $ the ideal generated by the elements
of $\LT_<(\I)$. A subset $G=\{p_1,\dots,p_s\}$ of $\I$ is said to be a \textit{Gr\"{o}bner
    basis} of $\I$ with respect to the monomial ordering $<$ if
\[ \left\langle\LT_<(p_1),\dots,\LT_<(p_s)\right\rangle=\left\langle\LT_<(\I)\right\rangle.\]
Note that a Gr\"{o}bner basis of $\I$ really is a generating set of the ideal $\I$, i.e. $\left\langle p_1,\dots,p_s\right\rangle=\I$ and every ideal in $\R[x]$ has a Gr\"{o}bner basis. 
A Gr\"{o}bner basis $G$ is \textit{reduced} if the two following conditions hold: i) the leading coefficient of $p$ is $1$, for all $p\in G$; ii) for all $p\in G$, there are no monomials of $p$ lying in $\left\langle \LT_<(G\setminus\{p\})\right\rangle$. Every ideal $\I$ has a unique reduced Gr\"{o}bner basis. We refer the reader to \cite{cox2013} for more details.

Let $\I$ be a zero-dimensional and radical ideal in $\R[x]$, and
$G$ be its reduced Gr\"{o}bner basis with respect to a
lexicographic ordering with $x_1<_{\lex} \cdots <_{\lex} x_n$. One says that $\I$ is
in \textit{shape position} if $G$ has the following shape:
\begin{equation}\label{f:sysshape}
    G=[w,x_2-u_2,\dots,x_n-u_n],
\end{equation}
where $w,u_2,\dots,u_n$ are polynomials in
$\R[x_1]$ and $\deg w=\# V_{\C}(\I)$.  
When the order $<_{\lex}$ is clear from the context, we omit the subscript $\lex$ in the above notation.

The following lemma, named Shape Lemma, gives us a criterion for the shape
position of an ideal.

\begin{lemma}{\rm{(Shape Lemma, \cite{gianni89})}}\label{lm:shape}
    Let $<_{\lex}$ be a lexicographic monomial order\-ing in $\R[x]$. Suppose that $\I$ is a zero-dimensional and radical ideal. If the points in $V_{\C}(\I)$  have distinct
    $x_1$-coordinates, then $\I$ can be made into
    shape position as in \eqref{f:sysshape}, where $u_2,\dots,u_n$ are
    polynomials in $\R[x_1]$ of degrees at most $\# V_{\C}(\I)-1$.
\end{lemma}

\section{The Squared Slack Variables Technique for General PCPs
}\label{sec:slack}
Using the squared slack variables technique, we show in this section that the solution set of $\CP(f)$ is the image of a real algebraic set under a certain projection.

Denote by $\I[f]$ the ideal in $\R[x]$ generated by the $n$ polynomial products $x_if_i$, $i\in[n]$, i.e.
\begin{equation}\label{eq:If}
    \I[f] \coloneqq\left\langle x_1f_1,\dots,x_nf_n \right\rangle.
\end{equation}
We introduce $2n$ squared slack variables $z=(z_1,\dots,z_{2n})$ and denote by $\I[f,z]$ the ideal in $\R[x,z]$ generated by the $n$ polynomials defining $\I[f]$, the $n$ quadratic polynomials $z_i^2-x_i$, $i\in[n]$, and the $n$ polynomials $z_{n+i}^2-f_i$, $i\in[n]$, i.e.,
\begin{equation}\label{eq:Ifz}
\I[f,z] = 
    \left\langle x_1f_1,\dots,x_nf_n, z_1^2-x_1,\dots, z_n^2-x_n, z_{n+1}^2-f_1,\dots, z_{2n}^2-f_n \right\rangle.
\end{equation}

Here and in the following, $\sqrt{c}$ stands for the set of the (two) complex square roots of a complex number $c$, and $\pi_n:\C^{n+m} \to \C^n$, for some non-negative integer $m$, is the projection onto the first $n$ coordinates.

\begin{remark}\label{rmk:z2}
Let $(\alpha,\beta)\in\C^{n+2n}$ with $\alpha=(\alpha_1,\dots,\alpha_{n})$. By the definition of $\I[f,z]$, it is not difficult to see that $(\alpha,\beta)$ is a point in $V_{\C}(\I[f,z])$ if and only if $\alpha \in \I[f]$ and $\beta$ is an element of the following set which is defined by a $2n$-fold Cartesian product:
\begin{equation}\label{eq:sqrt}
    \sqrt{\alpha_1} \times \dots \times \sqrt{\alpha_n} \times \sqrt{f_1(\alpha)} \times \dots \times \sqrt{f_n(\alpha)}.
\end{equation}
Furthermore, from the fact that $\#\sqrt{c}$ is finite, for $c\in \C$, we conclude that the set \textit{$V_{\C}(\I[f])$ is finite if any only if so is $V_{\C}(\I[f,z])$}.
\end{remark}

\begin{proposition}\label{prop:Sol_slack} One has the following equality:
    \begin{equation}\label{eq:Sol_slack}
        \Sol(f)=\pi_n\big(V_{\R}(\I[f,z])\big).
    \end{equation}
\end{proposition}
\begin{proof}Let $\alpha=(\alpha_1,\dots,\alpha_{n})\in\R^n$ be a point in $\Sol(f)$. Clearly,  the coordinates of $\alpha$ satisfy $\alpha_i f_i(\alpha)=0$, for all $i\in[n]$. Since $\alpha_i$ and $f_i(\alpha)$ are non-negative, we choose $\hat{z}_i \in \sqrt{\alpha_i}$ and $\hat{z}_{n+i} \in \sqrt{f_i(\alpha)}$ for $i\in[n]$. From \eqref{eq:Ifz}, $(\alpha,\hat{z})$ is a point in the real algebraic set of $\I[f,z]$. Thus, $\pi_n(\alpha,\hat{z})=\alpha$.

To prove the inverse inclusion, we suppose that $\alpha\in \pi_n\big(V_{\R}(\I[f,z])\big)$. There exists $\hat{z}\in\R^{2n}$ such that $(\alpha,\hat{z})\in V_{\R}(\I[f,z])$. By the definition of $\I[f,z]$, one has  $\alpha_i f_i(\alpha)=0$, $\alpha_i=\hat{z}_i^2\geq 0$, and $f_i(\alpha)=\hat{z}_{n+i}^2\geq 0$, for all $i\in[n]$. Hence, $\alpha$ is a element of $\Sol(f)$.

The equality \eqref{eq:Sol_slack} is proved. \qed
\end{proof}

\begin{remark}
From \Cref{prop:Sol_slack}, the problem $\CP(f)$ is solvable if and only if the polynomials defining $\I[f,z]$ have common real roots. In some computer algebra systems, we can check whether a system of polynomials has real roots. In {\sc Maple}, for example, we can use $\mathtt{HasRealRoots}$ in the $\mathtt{RootFinding}$ package to perform such tasks for polynomials with rational coefficients. 
\end{remark}

\begin{example}\label{ex:CP1} Consider a linear complementarity problem $\CP(f)$ in two variables, where
    $f(x_1,x_2)=(x_2-1,x_1-1)$. We use four slack variables $(z_1,\dots,z_4)$ and consider the following system of polynomials in $\R[x_1,x_2,z_1,z_2,z_3,z_4]$:
\[\{x_1x_2-x_1, \ x_1x_2-x_2, z_1^2-x_1,z_2^2-x_2,z_3^2-x_2+1,z_4^2-x_1+1  \}.\]
This system has real roots, for example, $(1,1,1,1,0,0)$. We conclude that the complementarity problem has a solution.
\end{example}

\begin{remark}\label{rmk:slack}
The squared slack variables technique has been demonstrated to be an effective method for converting polynomial inequalities into equations. In particular, every basic closed semi-algebraic set $S$ in $\R^n$ given by the inequality constraints $\{g_1(x)\geq 0,\dots, g_m(x)\geq 0\}$ is the image of the following real algebraic set under projection $\pi_n$:
\[\{(x,z)\in\R^{n+m} : z_1^2-g_1(x)=0,\dots, z_m^2-g_m(x)=0 \}.\]
Consequently, findings on real algebraic sets, e.g., compactness or computing dimension \cite{pham2018compactness,lairez2021computing}, can be extended to closed semi-algebraic sets. A limitation of this approach is that the number of variables in the resulting system is $n+m$. However, the degrees of the new variables are only two. 
\end{remark}

\section{Existence of Univariate Representations for $\D_0$ Problems}\label{sec:CP}
 In this section, we focus on the case where $\I[f]$ is zero-dimensional, i.e. the complex variety $V_{\C}(\I[f])$ has finitely many points. It should be noted that this condition can be checked by using computer algebra systems.

 \begin{definition} If the complex variety $V_{\C}(\I[f])$ has finitely many points, we say that the polynomial complementarity problem given by $f$ is a $\D_0$ problem, or $f$ is $\D_0$ for simplicity.
 \end{definition}

Let $\R_d[x]$ be the set of all polynomials with real coefficients of degree at most $d$. 
Recall that the dimension of $\R_d[x]$, denoted by $\rho$,
is finite and the linear space $\R_d[x]^n$ is isomorphic
to $\R^{n\times \rho}$.
For every $f=(f_1,\dots,f_n)$ in $\R_d[x]^n$, $f_i$ can be written (uniquely) with coefficients $a_i \in\R,\ Q_i \in \R^{\rho-1}$ as follows:
$f_i(x)=a_i+ Q_i^T \vecx(x)$, 
where 
\begin{equation}\label{eq:X}
\vecx(x)=(x_1,\dots,x_n,x_1^2,x_1x_2,\dots,x_n^2,\dots,x_n^d)^T.
\end{equation}
Then, $f$ is written concisely $f(x)=a+ Q^T \vecx(x)$, where $a=(a_1,\dots,a_n)$ and $Q=(Q_1,\dots,Q_n)$.

\subsection{Genericity of $\D_0$ Problems} 
A subset $\mathcal{O} \subset \R^m$ is said to be \textit{generic} on $\R^m$ 
 if its complement $\R^m\setminus \mathcal{O}$ has Lebesgue measure zero. 
Recall from \cite[Proposition 1.5]{pham2016genericity} that if  $\mathcal{O} \subset \R^m$ is an open and dense semi-algebraic subset, then $\mathcal{O}$ is generic in  $\R^m$.

\hieu{This subsection aims at proving that the property of being $\D_0$ is generic, i.e. there exists a generic subset $\mathcal{O}$ in the parameter space $\R^{n\times \rho}$ such that $a+ Q^T \vecx(x)$ is $\D_0$ for every $(a,Q)\in \mathcal{O}$.}
 
\hieu{The following theorem (\Cref{thm:Sard_para}) is a version of Sard's theorem with real parameters.
Other versions can be found, for example, in  \cite[Theorem 1.10]{pham2016genericity} for the semi-algebraic setting (with real parameters), or in \cite[Proposition B.3]{safey2017} for the algebraic setting (with complex parameters). The claim is considered folklore; however, a suitable reference for it has yet to be found. The proof is postponed to the appendix.
The notions of real/complex manifolds and regular values can be found in monographs, for example, \cite{lee2003smooth,lee2024complex}.
}
\hieu{
\begin{theorem}\label{thm:Sard_para}
Let $V\subset \C^n$ be a smooth manifold, and $\Phi:\R^m\times V \to \C^n$ be a polynomial mapping. Assume that $V$ is the difference of two algebraic sets and $0\in\C^n$
is a regular value of $\Phi $. Then, there exists an open and dense semi-algebraic set $\mathcal{O}$ in $\R^m$ such that $0$ is a regular value of the mapping 
\[\Phi_{c}: V \to \C^{n}, \ \Phi_{c}(x)=\Phi(c,x),\] 
for any $c \in \mathcal{O}$.
\end{theorem}
}

\begin{proposition}\label{prop:genericity} 
There exists an open and dense semi-algebraic subset $\mathcal{O}$ in $\R^{n\times \rho}$ such that $f(x)=a+ Q^T \vecx(x)$ is $\D_0$ for every $(a,Q)\in \mathcal{O}$.
 \end{proposition}

\begin{proof}
Let $\ell$ be a subset of $[n]$ and $\mathcal{V}_{\ell}$ be the set of all $\alpha\in\C^n$ such that $\alpha_i\neq 0$ if $i\in \ell$ and $\alpha_i = 0$ if $i\in\bar\ell=[n]\setminus \ell$. Clearly, $\mathcal{V}_{\emptyset} = \{0\}$ and $\C^n$ is the disjoint union of $\mathcal{V}_{\ell}$ for $\ell \subseteq [n]$.

Let $\ell\neq \emptyset$ be given. Without loss of generality, we assume that $\ell=\{1,\dots,l\}$. For convenience, we write $x_{\ell}=(x_1,\dots,x_{l})$ and $0=(0_{\ell},0_{\bar\ell})$, where  $0_{\bar\ell}$ is the zero $(n-l)$-vector.
Then, the set $\mathcal{V}_{\ell}$ can be considered as the difference of two algebraic sets:
\[\mathcal{V}_{\ell}=\C^l\times \{0_{\bar\ell}\} \setminus \{(x_{\ell},0_{\bar\ell}): x_1 \cdots x_l = 0\}.\]

We consider the polynomial mapping 
\[\Phi_{\ell}:\R^{n}\times \R^{n \times (\rho-1)}\times \C^l\times \{0_{\bar\ell}\} \to \C^{ l }\]
defined by
\[\Phi_{\ell}(a,Q,x_{\ell},0_{\bar\ell}) = \Big(a_1x_1+x_1Q_1^T \vecx(x_{\ell},0_{\bar\ell}), \ \dots, \ a_lx_l+x_lQ_l^T \vecx(x_{\ell},0_{\bar\ell})\Big),\]
where $Q \in \R^{(\rho-1)\times n}$ and $\vecx(x)$ is given in \eqref{eq:X} with $x=(x_{\ell},0_{\bar\ell})$.
One sees that the restriction 
\[\Phi_{\ell}:\R^{n}\times \R^{n \times (\rho-1)}\times \mathcal{V}_{\ell} \to \C^{ l }\]
is differentiable. 
Obviously, its Jacobian matrix with respect to $(a,Q,x_{\ell})$, denoted by $J_{\ell}$, has size $l\times(n \rho +l)$. Moreover, $J_{\ell}$ has a has the diagonal submatrix $\partial_a \Phi_\ell$, whose entries are the components of $x_\ell$. Since all coordinates of $x_{\ell}$ are non-zero, we conclude that  $J_{\ell}$ 
 is full rank. It follows that the origin $0_{\ell}$ in $\C^{ l }$ is a regular value of $\Phi_{\ell}$. By Theorem \ref{thm:Sard_para}, there exists a generic set $\mathcal{O}_{\ell}$ in $\R^{n\times \rho}$ such that for every $(a,Q)$ in $\mathcal{O}_{\ell}$, $0_{\ell}$ is a regular value of 
 $\Phi_{\ell,a,Q}: \mathcal{V}_{\ell} \to \C^{ l }$ 
 defined by 
 $\Phi_{\ell,a,Q}(x_{\ell},0_{\bar\ell})=\Phi_{\ell}(a,Q,x_{\ell},0_{\bar\ell})$.
It is clear that 
$$\Phi_{\ell,a,Q}^{-1}(0_{\ell})=\Big\{(x_{\ell},0_{\bar\ell})\in \mathcal{V}_{\ell}: a_ix_i+ x_iQ_i^T \vecx(x_{\ell},0_{\bar\ell})  = 0, i\in\ell \Big\},$$ 
and according to the Regular Level Set Theorem \cite[Chapter 2]{lee2024complex}, this level set is a complex sub-manifold whose dimension is $0$ or it is empty. Therefore, this set has finitely many point. 

Finally, we observe that 
\[V_{\C}(\I[f]) \subseteq \big\{0\big\} \cup \Big(\bigcup_{\ell\subseteq [n], \ell \neq \emptyset} \Phi_{\ell,a,Q}^{-1}(0_{\ell})\Big),\]
and that the finite union $\mathcal{O}:= \bigcup_{\ell\subseteq [n], \ell \neq \emptyset} \mathcal{O}_{\ell}$ is generic in $\R^{n\times \rho}$. Hence, for every $(a,Q)$ in $\mathcal{O}$, $V_{\C}(\I[f])$ has finitely many points.
\qed
\end{proof}

\subsection{A Sharp Bound on the Cardinality of $V_{\C}(\I[f,z])$}

It should be highlighted that if we use $m=2n$ squared slack variables as in \Cref{rmk:slack}, the cardinality of the complex algebraic set can increase \hieu{up to $2^{2n}$ times the original}. However, in the setting of the complementarity problem, the \hieu{last number is reduced from $2^{2n}$ to $2^{n}$}, as shown in the following proposition. Recall that the degree of the polynomial mapping $f$ is the following number:
$$d:=\max\{\deg f_i:i\in [n]\}.$$

\begin{proposition}\label{prop:card} Let $f$ be $\D_0$. Then, the following inequalities hold:
\begin{equation}\label{eq:cardV}
\# V_{\C}(\I[f,z]) \ \leq  \ 2^n \times \#  V_{\C}(\I[f]) \ \leq \ 2^n(d+1)^n.
\end{equation}
\end{proposition}
\begin{proof} We prove that for a given solution $\alpha$ the cardinality of the set  \eqref{eq:sqrt} does not exceed $2^n$. Indeed, for each pair $\{\alpha_i, f_i(\alpha)\}$, there is at least one entry that is zero, so the cardinality of $\sqrt{\alpha_i} \times\sqrt{f_i(\alpha)}$ is at most two. Hence, the cardinality of the set \eqref{eq:sqrt} does not exceed $2^n$. This implies that \[\# V_{\C}(\I[f,z]) \ \leq \ 2^n \times \#  V_{\C}(\I[f]).\]

Applying Bézout's theorem  \cite[Chapter 8, \S 7]{cox2013} to the system defining $\I[f]$, one has $\# V_{\C}(\I[f])\leq (d+1)^n$. Therefore, the second inequality in \eqref{eq:cardV} holds.
    \qed
\end{proof}

We are also interested in the maximum number of solutions of $\D_0$ problems with fixed degree $d$ and number of variables $n$. A sharp upper bound on this number is given in the following remark.

\begin{remark}\label{rmk:degTCP2} Let $d$ be fixed. Then the maximum number of solutions of $\CP(f)$, where $f$ is of degree $d$ and is $\D_0$, does not exceed $(d+1)^{n}$. Indeed, it is clear that $\#\Sol(f)\leq \# V_{\C}(\I[f])\leq (d+1)^{n}$, where the last number is the Bézout bound of $V_{\C}(\I[f])$. The maximum is reached, i.e., there exists $p\in\R[x]^n$ of degree $d$ and being $\D_0$ such that $\#\Sol(p)=(d+1)^{n}$. We consider the polynomial map $p$ given by
  \[  \hieu{p_i(x)=(-1)^{d}\prod_{j=1}^{d}(x_i-j), \ i \in [n]}.\]
Clearly, $x_ip_i$ has $d+1$ non-negative (real) roots, listed by $\{0,1,\dots,d\}$. It is not difficult to check that
    \begin{center}
        $V_{\C}(\I[p])=V_{\R}(\I[p]) = \big\{(a_1,\dots,a_n)\in\R^n: a_i\in \{0,1,\dots,d\}\big\}$,
    \end{center}
    and that the solution set $\Sol(p)$ coincides with $V_{\R}(\I[p])$. Hence, the number of points in $\Sol(p)$ is $(d+1)^{n}$.
\end{remark}

\begin{remark}\label{rmk:takeda} It should be noted that the number of variables is generally the most important factor in the complexity of symbolic computation on polynomials. We remark here about a technique for reducing the number of slack variables. It is logical to replace $x_i$, which is required to be non-negative, by $y_i^2$ in $f(x)$, $i\in [n]$. Then, the system \eqref{eq:fxzero} becomes the following system:
\[f(y^2)\geq 0,\  y_1f_1(y^2)=0, \ \dots, \ y_nf_n(y^2)=0,\]
where $y^2:=(y_1^2,\dots,y_n^2)$. We can add $n$ squared slack variables to the above system (regarding the condition $f(y^2)\geq 0$), then we get a new system of $2n$ variables. However, the degree of this system could be $2d+1$, then the cardinality of the complex variety can reach $2^n (2d+1)^{n}$. This number is larger than the bound $2^n(d+1)^{n}$ given in \eqref{eq:cardV}.
\end{remark}
 
\subsection{Univariate Representations of the Solutions}

\begin{theorem}\label{thm:solCP} Let $f$ be $\D_0$. Then, there exist univariate polynomials $v_1,\dots,v_n$, $ w$ in $\R[t]$, with $\deg v_i<\deg w$, $i\in[n]$, and $\deg w=\#V_{\C}(\I[f,z])$,  such that
\begin{equation}\label{eq:univrep}
\Sol(f)=\big\{\big(v_1(t),\dots,v_n(t)\big): w(t)=0, t\in\R\big\}.
\end{equation}
\end{theorem}

\begin{proof} Since $V_{\C}(\I[f])$ is finite, according to \cite[Lemma 2.1]{rouillier1999}, there exists an invertible $3n \times 3n$-matrix $H$ with real entries such that, if we change the variables by $(x,z)=Hy$ in the following polynomials
\begin{equation}\label{eq:listfz}
    x_1f_1,\dots,x_nf_n, z_1^2-x_1,\dots, z_n^2-x_n, z_{n+1}^2-f_1,\dots, z_{2n}^2-f_n,
\end{equation}
we obtain respectively polynomials $h_1(y),\dots,h_{3n}(y)$, and introduce in $\R[y_1,\dots,y_{3n}]$ the ideal  
\[\I_h\coloneqq\left\langle h_1,\dots,h_{3n} \right\rangle,\]
then the points in $V_{\C}(\I_h )$ have distinct $y_1$-coordi\-nates. Thanks to \Cref{lm:radC}, we have
$V_{\C}(\rad(\I_h ))= V_{\C}(\I_h )$. 
In view of the assumption and \Cref{prop:card}, $\rad(\I_h )$ is zero-dimensional and radical.

Clearly, the ideal $\rad(\I_h )$ satisfies the condition of the Shape Lemma. So, there exist $3n$ polynomials in $\R[y_1]$, say $w$, $u_2,\dots,u_{3n}$, with  $\deg w=\#V_{\C}(\I[f,z])$ and
$\deg u_i<\deg w$, $i=2,\dots,3n$, and
such that
$[w,y_2-u_2,\dots,y_{3n}-u_{3n}]$ is the reduced Gr\"{o}bner basis of $\rad(\I_h )$ with respect to a lexicographic monomial ordering $y_1< \cdots < y_{3n}$.

We write $u(t):=(t,u_2(t),\dots,u_{3n}(t))$ for short. From the change of variables, one has
\begin{equation*}\label{eq:VCf}
V_{\C}(\I[f,z])=\big\{H(u(t))^T: w(t)=0\big\}.
\end{equation*}
Recall that the entries of $H$ are real, so that the ones of $H^{-1}$ are so as well. Suppose that $\tau$ is a root of $w(t)$. If $\tau$ is real, then all coordinates of $H(u(\tau))^T$ are real, hence this vector belongs to $V_{\R}(\I[f,z])$. Conversely, if $H(u(\tau))^T$ is in $\R^{3n}$ then
\[u(\tau)^T= \left(H^{-1}H\right) (u(\tau))^T\in\R^{3n}.\]
Hence, $\tau$ belongs to $\R$. We conclude that the real algebraic variety $\I[f,z]$ is given by
\begin{equation}\label{eq:VRf}
    V_{\R}(\I[f,z])=\big\{H(u(t))^T: w(t)=0, t\in\R\big\}.
\end{equation}

Let $v_1(t),\dots,v_n(t)$ be the first $n$ components of $ H(u(t))$.
Since the polynomials $u_2,\dots,u_{3n}$ have real coefficients, so do $v_1,\dots,v_n$. 
Moreover, if $\deg v_i\geq \deg w$ then we replace $v_i$ by the remainder of the division of $v_i$ by $w$, for $i\in[n]$.

The conclusion of the theorem follows from \eqref{eq:VRf} and \eqref{eq:Sol_slack}. \qed
\end{proof}

The system $[w,v_1,\dots,v_n] \subset \R[t]$ mentioned in \Cref{thm:solCP} is univariate and such a system is called a \textit{univariate representation} of the solutions 
of $\CP(f)$.

To illustrate Theorem \ref{thm:solCP} and its proof, we consider the following example.

\begin{example}\label{ex:CP2} We consider the complementarity problem given in \Cref{ex:CP1}. 
    From \eqref{eq:If}, the ideal $\I[f]$ is given by
$\I[f]= \left\langle x_1x_2-x_1, \ x_1x_2-x_2 \right\rangle$.
This ideal is zero-dimensional as its complex variety contains two distinct points,
$V_{\C}(\I[f])=\left\lbrace (0,0),(1,1)\right\rbrace$.
We use four slack variables $z=(z_1,\dots,z_4)$ and consider the following ideal in $\R[x,z]$:
\[\I[f,z]= \left\langle x_1x_2-x_1, \ x_1x_2-x_2, z_1^2-x_1,z_2^2-x_2,z_3^2-x_2+1,z_4^2-x_1+1  \right\rangle.\]
By changing variables $(x,z)=Hy$, where $H$ is a $6\times 6$-matrix as follows:
\[H:=\begin{bmatrix}
   1 & 2 & \cdots & 6 \\
   0 & 1 & \cdots & 0  \\
   \vdots & \vdots  & \ddots & \vdots \\
   0 & 0 & \cdots & 1
\end{bmatrix},\]
we obtain six polynomials $h_1,\dots,h_6$ in variables $y$. The ideal $\rad(\I_h) $ is in shape position with the pure lexicographic ordering $y_1<y_2<\cdots<y_6$. The Gr\"{o}bner basis of $\rad(\I_h) $ is $[w(y_1),y_2-u_2(y_1),\dots,y_6-u_6(y_1)]$. The first two coordinates of $H(t,u_2,\dots,u_6)$ are denoted by $v_1,v_2$. So we obtain $w$ and  $v_1,v_2$ as follows:

{\small $w(t)=t^{8}+4 t^{7}+78 t^{6}+392 t^{5}-5247 t^{4}-11228 t^{3}-5324 t^{2}-11616 t$}, 
{\small $v_1(t) =(-395708 t^{7} - 1758766 t^{6} - 31180792 t^{5} - 167371327 t^{4}+ 2037780580 t^{3}+ 5545360564 t^{2} + 2068565664 t + 5710973125)/5710973125$},
{\small \hieu{$v_2(t) = v_1(t)$}}.

Polynomial $w$ has eight roots $-8, -2, 0, 6,  \mathrm{i}, - \mathrm{i},11 \mathrm{i}, -11 \mathrm{i}$, where $\mathrm{i}$ is the imaginary unit. According to \eqref{eq:univrep}, the solution set is given by
\[\Sol(f)=\big\lbrace (v_1(t),v_2(t)):t\in\{-8, -2,0, 6\}\big\rbrace  = \{(1,1)\}.\]
\end{example}

\subsection{Univariate Representations of Least-norm Solutions and Sparse Solutions}
Once a univariate representation for the solutions of a $\D_0$ problem is obtained, we can provide a univariate representation for the set of least-norm solutions $\Solln(f)$. In particular, we have the following theorem:

\begin{theorem}\label{thm:sol_least} Let $f$ be $\D_0$ and $[w,v_1,\dots,v_n]$ be a univariate representation of $\Sol(f)$. Then, there exists $\varphi$ in $\R[t]$ such that
\begin{equation}\label{eq:univ_least}
\Solln(f)=\Big\{\big(v_1(t),\dots,v_n(t)\big): t\in\arg\min\{\varphi(t): w(t)=0, t\in\R\}\Big\}.
\end{equation}
\end{theorem}

\begin{proof}
It is clear that the following real function 
\begin{equation}\label{eq:phi}
 \hat{\varphi}(t):=v_1^2(t)+\dots+v_n^2(t), \ t\in\R, 
\end{equation}
is a univariate polynomial with real coefficients. Denote by $\varphi$ the remainder of the division of $\hat{\varphi}$ by $w$. One has $\deg\varphi < \deg w$.
 Let $\gamma$ be the minimum of $\varphi(t)$ over the finite set $\{t\in\R, w(t)=0\}$. It is not difficult to prove that  $\hat{\alpha}$ is a least-norm solution of $\CP(f)$ if and only if there exists a real root $\tau$ of $w$ such that $\varphi(\tau)=\gamma$ and  $\hat{\alpha}=(v_1(\tau),\dots,v_n(\tau))$. Hence, \eqref{eq:univ_least} is proved.
 \qed
\end{proof}

The results concerning least-norm solutions can be extended to other norms. However, in this paper with the polynomial setting, we only consider the Euclidean norm $\|\cdot\|$.

Let $[w,v_1,\dots,v_n]$ be a univariate representation of $\Sol(f)$ and $\alpha$ be a solution of $\CP(f)$. Let $\ell$ be a subset of $[n]$. We denoted $w_{\ell}$ by the greatest common divisor defined as follows:
$$w_{\ell}:=\gcd(w,v_{i}, i\in \ell).$$

\begin{theorem}\label{thm:sol_sp} Let $f$ be $\D_0$ and $[w,v_1,\dots,v_n]$ be a univariate representation of $\Sol(f)$. Suppose that sparse solutions have $k$ zero coordinates. Then
\begin{equation}\label{eq:univ_sp}
\Solsp(f)=\bigcup_{\ell\subseteq [n], k= l }\Big\{\big(v_1(t),\dots,v_n(t)\big): w_{\ell}(t)=0, w_{\bar\ell}(t)\neq 0, t\in\R\Big\}.
\end{equation}
\end{theorem}

\begin{proof} Let $\alpha$ be a sparse solution. Suppose that $\ell$ is the index set of the zero coordinates of $\alpha$. So, there are some real roots $\tau$ of $w$ such that $\alpha=\big(v_1(\tau),\dots,v_n(\tau)\big)$ and $v_i(\tau)=0$ for $i\in\ell$, and the rest of the coordinates $v_i(\tau)\neq 0$ for $i\in\bar\ell$. Hence $\tau$  is a real root of $w_{\ell}$ and $\tau$ is not a real root of $w_{\bar\ell}$. Therefore, the inclusion $\subseteq$ of \eqref{eq:univ_sp} is proved. The converse
inclusion is similarly demonstrated. 
\qed
\end{proof}

\section{Algorithms Computing Univariate Representations}\label{sec:compRn}
Based on the proof of Theorem \ref{thm:solCP}, we first design an algorithm (\Cref{alg:Sol}) to compute a univariate representation of the solution set of a $\D_0$ problem.  We then briefly describe algorithms for computing univariate representations for least-norm and sparse solutions based on \Cref{alg:Sol}.

\subsection{Computing Univariate Representations of $\Sol(f)$}

\textit{Input and output}. The input of \Cref{alg:Sol} is a polynomial mapping $f$ being $\D_0$.
The output is a list of univariate polynomials $[w,v_1,\dots,v_n]\subset \R[t]$ such that $\Sol(f)$ can be represented as in \eqref{eq:SolAa}.

\begin{algorithm}
    \caption{Computing a univariate representation of $\Sol(f)$}\label{alg:Sol}
    \smallskip

    \textbf{Input:} $f\in\R[x]^n$ such that $f$ is $\D_0$
    \smallskip

    \textbf{Output:} $[w,v_1,\dots,v_n]\subset \R[t]$ - a univariate representation of $\Sol(f)$,
    \begin{equation}\label{eq:SolAa}
        \Sol(f)=\Big\{\big(v_1(t),\dots,v_n(t)\big): w(t)=0, t\in\R\Big\},
    \end{equation}
where $\deg w=\#V_{\C}(\I[f,z])$ and $\deg v_i<\deg w$, $i\in[n]$
    \begin{itemize}
        \item [\rm 1:] Set $s\coloneqq0$ and $\shape\coloneqq$ false
        \smallskip
        \item [\rm 2:] While $\shape=$ false do
        \smallskip
        \begin{itemize}
            \item [\rm 3:] Change variables $(x,z)=Hy$ in the polynomials in \eqref{eq:listfz}, respectively, obtain polynomials denoted by $h_i(y)$, $i\in[3n]$, where
            \begin{equation}\label{eq:H}
                H:=\begin{bmatrix}
                    1 & s & s^2 &\cdots & s^{3n-1} \\
                    0 & 1 & 0 & \cdots  & 0 \\
                    0& 0 & 1 & \cdots & 0\\
                    \vdots & \vdots & \vdots & \ddots & \vdots\\
                    0& 0 & 0 & \cdots & 1
                \end{bmatrix}^{-1} \in\R^{3n \times 3n}
            \end{equation}
            \item [\rm 4:] Set $\I_h \coloneqq\left\langle h_1,\dots,h_{3n} \right\rangle$ in $\R[y_1,\dots,y_{3n}]$ and compute the reduced Gr\"{o}bner basis $G$ of $\rad(\I_h )$ with respect to the pure lexicographic ordering $y_1< y_2< \cdots < y_{3n}$
            \smallskip
            \item [\rm 5:] If the basis $G$ has the following form \begin{equation}\label{eq:G}
                G=[ w,y_2-u_2,\dots,y_{3n}-u_{3n}],
            \end{equation}
where $ w,u_2,\dots,u_{3n}$ are in $\R[y_1]$, i.e. $\rad(\I_h) $
            is in shape position,
then set $\shape\coloneqq$ true, else
            set $s\coloneqq s+1$
            \smallskip
        \end{itemize}
        \item [\rm 6:]
Change the variable $y_1$ by $t$ and define by $(v_1,\dots,v_n)$ the first $n$ coordinates of $H(t,u_2,\dots,u_{3n})^T$ 
\smallskip
\item [\rm 7:] Return $[w,v_1,\dots,v_n]$
    \end{itemize}
\end{algorithm}

\noindent \textit{Description}. At the beginning, we set $s=0$ and $\shape=$ false. The variable $\shape$ is boolean.
The while loop in Line~2 will stop if $\shape$ is true. Within the while loop, in Line 3, with $s$ given, we define matrix $H$  depending on $s$ as in \eqref{eq:H} and change the variables $(x,z)=Hy$ in the polynomials in \eqref{eq:listfz} and obtain $h_i$, $i\in[3n]$.  In Line~4,
we set
\[\I_h \coloneqq\left\langle h_1,\dots,h_{3n} \right\rangle\subset \R[y_1,\dots,y_{3n}]\]
and compute the reduced Gr\"{o}bner basis $G$ of the radical of the ideal $\I_h $, that is $\rad(\I_h)$, with respect to the pure lexicographic order with $y_1< y_2< \cdots < y_{3n}$. In Line 5, we set the value of $\shape$ as true if 
$\rad(\I_h)$ is in shape position
otherwise, we increase the value $s\coloneqq s+1$.

When the while loop ends, the reduced Gr\"{o}bner basis $G$ has the form \eqref{eq:G}, then $\rad(\I_h)$ is in shape position
 with \[\rad\left( \I_h\right) =\left\langle w,y_2-u_2,\dots,y_{3n}-u_{3n}\right\rangle.\]

The correctness of \Cref{alg:Sol} is proved in the following proposition.

\begin{theorem}\label{prop:corectCP} Let $f$ be $\D_0$. On input $f$,
\Cref{alg:Sol} terminates and returns a univariate representation of the solution set $\Sol(f)$.
\end{theorem}
\begin{proof} Denote by $\delta$ the cardinality of $V_{\C}(\I[f])$, according to \cite[Lemma 2.1]{rouillier1999},
there is a number $s$ in the set $\{0,\dots,(3n-1)\delta(\delta-1)/2\}$ such that the linear function
\[u(x,z)\coloneqq x_1+ \dots +s^{n-1}x_n+s^nz_1+\dots +s^{3n-1}z_{2n} \]
separates $V_{\C}(\I[f,z])$, i.e., $u(a)\neq
    u(b)$ for any distinct points $a,b$ in $V_{\C}(\I[f,z])$. Hence, the while loop ends within at most $(3n-1)\delta(\delta-1)/2+1$ iterations.

After the while loop in Line 2, the reduced Gr\"{o}bner basis of $\rad(\I_h )$, written as in Line~5, is
    $G=[w,y_2-u_2,\dots,y_n-u_{3n}]$.
After Line 6, one has
$$\big\{\big(v_1(t),\dots,v_n(t)\big): w(t)=0, t\in\R\big\}=\pi_n(V_{\R}(\I[f,z])).$$
From the latter equality and \eqref{eq:Sol_slack}, we get \eqref{eq:SolAa}.
Therefore, we claim that the output of \Cref{alg:Sol} is a univariate representation of $\CP(f)$.
\qed
\end{proof}

\begin{remark}
As seen in the proof of \Cref{prop:corectCP}, the while loop is executed at most $(3n-1)\delta(\delta-1)/2+1$ times, where $\delta=\# V_{\C}(\I[f,z])$. By \Cref{prop:card} we get an upper bound on the number of iterations of the loop by \[(3n-1)2^n(d+1)^n(2^n(d+1)^n-1)/2+1.\]
However, in practice the matrix $H$ should be chosen randomly with rational entries, hence we do not need to use the while loop in Line 2.
\end{remark}

\subsection{Computing Univariate Representations of $\Solln(f)$}

A univariate representation of $\Solln(f)$ can be computed based on the previous algorithm. To complete this task, we propose \Cref{alg:Solln}, but here we briefly describe it and skip the proof of its correctness. Once a univariate representation $[w,v_1,\dots,v_n]$ of $\Sol(f)$ obtained as the output of \Cref{alg:Sol}, as mentioned in the proof of \Cref{thm:sol_least}, we compute $\hat{\varphi}$ as in \eqref{eq:phi} and then $\varphi$ is the remainder of the division of $\hat{\varphi}$ by $w$.  

\begin{algorithm}
    \caption{Computing a univariate representation of $\Solln(f)$}\label{alg:Solln}
    \smallskip

    \textbf{Input:} \hieu{$f\in\R[x]^n$ such that $f$ is $\D_0$}
    \smallskip

    \textbf{Output:} $[w,v_1,\dots,v_n,\varphi]\subset \R[t]$ - a univariate representation of $\Solln(f)$ satisfying \eqref{eq:univ_least}, i.e.,
\[\Solln(f)=\big\{\big(v_1(t),\dots,v_n(t)\big): t\in\arg\min\{\varphi(t): w(t)=0, t\in\R\}\big\}\]
    \begin{itemize}
    \smallskip
        \item [\rm 1:]     
\hieu{Compute $[w,v_1,\dots,v_n]\subset \R[t]$ - a univariate representation of $\Sol(f)$}
    \smallskip
        
       \item [\rm 2:]  Compute $\hat{\varphi}(t)=v_1^2(t)+\dots+v_n^2(t)$ and $\varphi$ - the remainder of the division of $\hat{\varphi}$ by $w$
\smallskip
\item [\rm 3:] Return $[w,v_1,\dots,v_n,\varphi]$
    \end{itemize}
\end{algorithm}

\subsection{Computing Univariate Representations of $\Solsp(f)$}
We design \Cref{alg:Solsp} to compute a univariate representation of the sparse solution set of a $\D_0$ problem. 

\noindent \textit{Description}. The output of \Cref{alg:Solsp} consists of a subset $\Omega\subset 2^{[n]}$ and $[w,v_1,\dots,v_n]\subset \R[t]$ such that, for every element $\ell$ in $\Omega$, the following set is nonempty: 
\begin{equation}\label{eq:t_nonempty}
\{t\in\R:w_{\ell}(t)=0, w_{\bar\ell}(t)\neq 0\},
\end{equation}
and for each real number $\tau$ in the latter set, $(v_1(\tau),\dots, v_n(t))$ is a sparse solution whose coordinates at $\ell$ are zero.

\begin{algorithm}
    \caption{Computing a univariate representation of $\Solsp(f)$}\label{alg:Solsp}
    \smallskip

    \textbf{Input:} \hieu{$f\in\R[x]^n$ such that $f$ is $\D_0$}
    \smallskip

    \textbf{Output:} $\Omega$ and $[w,v_1,\dots,v_n]$ such that
\begin{equation}\label{eq:univ_sp2}
\Solsp(f)=\bigcup_{\ell\in \Omega}\Big\{\big(v_1(t),\dots,v_n(t)\big): \{t\in\R:w_{\ell}(t)=0, w_{\bar\ell}(t)\neq 0\}\neq \emptyset\Big\}.
\end{equation}
    
\begin{itemize}
\item [\rm 1:]     
\hieu{Compute $[w,v_1,\dots,v_n]\subset \R[t]$ - a univariate representation of $\Sol(f)$}
    \smallskip
\item [\rm 2:] Set $k:=0$ and $\Omega:=\emptyset$ 
\smallskip
\item [\rm 3:] For $j$ from $n$ to $1$ do
\smallskip
\begin{itemize}
    \item [\rm 4:] For each $\ell\subseteq [n]$ such that $ \#\ell  = j$, if $\{t\in\R:w_{\ell}(t)=0, w_{\bar\ell}(t)\neq 0\}\neq \emptyset$, then set $k:=j$ and $\Omega:=\Omega \cup \ell$
    \smallskip
    \item [\rm 5:] If $k>0$, then stop the for loop in Line 3
        \end{itemize}
\smallskip
\item [\rm 6:] Return $\Omega$ and $[w,v_1,\dots,v_n]$
    \end{itemize}
\end{algorithm}

In \Cref{alg:Solsp}, $k$ indicates the number of zero coordinates of a sparse solution, and $\Omega \subset 2^{[n]}$ is displayed as the set of all $\ell\subset [n]$ with $\#\ell  = k$ such that there is a sparse solution whose
coordinates at $\ell$ are zero. In the initial stage, we set $k=0$ and $\Omega=\emptyset$, which indicates that the problem lacks sparse solutions. 

Since $k$ belongs to $\{0,1,\dots,n\}$ and is the maximum number of zero coordinates, in Line 3 we reduce the value of $j$ starting from $n$. The for loop is stopped when the first index set $\ell$, of a sparse solution, is found. We note that, for each $j$, there are $n!/(j!(n-j)!)$ subsets of $[n]$ with cardinality $j$. Once a sparse solution is found, in Line 4, we for each $\ell\subseteq [n]$ with $k$ elements, if $w_{\ell}$ has real roots and $w_{\bar\ell}$ has no real roots, add this set to $\Omega$.

\noindent \textit{Correctness}. The correctness of \Cref{alg:Solsp} is briefly explained as follows: After the for loop in Line 3, $k$ is the value of $\|\alpha\|_0$ for a sparse solution $\alpha$. Clearly, $\ell\in \Omega$ if and only if the set \eqref{eq:t_nonempty} is non-empty. Furthermore
for each $\tau$ in the set \eqref{eq:t_nonempty}, $(v_1(\tau),\dots, v_n(\tau))$ is a sparse solution whose coordinates at $\ell$ are zero. Hence, we get the equality \eqref{eq:univ_sp2}.

\begin{remark}
The Shape Lemma still holds if we replace the field of real numbers $\R$ with the field of rational numbers $\Q$, i.e. if $\I$ is an ideal in the ring $\Q[x]$ and satisfies the conditions in \Cref{lm:shape} then the involved polynomials $w,u_2,\dots, u_n$ have rational coefficients.
Therefore, if the input polynomial mapping $f$ has rational coefficients, then the univariate polynomials in the aforementioned univariate representations also have rational coefficients.
\end{remark}

\section{Computing Numerical Solutions for $\D_0$ Problems}\label{sec:enum}
This section deals with enumerating (least-norm, sparse) solutions, i.e. explicitly computing all such solutions, 
of a polynomial complementarity problem $\CP(f)$, where $f$ is $\D_0$.

\subsection{Enumerating Solutions}\label{subsec:enumerate}

It is worth to remark that the number of real roots of $w$ can be greater than the cardinality of $\Sol(f)$; in particular, from the proof of \Cref{prop:card}, one has 
$$\#\{t\in\R:w(t)=0\}\le 2^n\times\#\Sol(f).$$
The equality can be reached. For example, in \Cref{ex:CP1}, $w$ has four real roots, says $-8,-2,0,6$, while the problem has only one solution $(1,1)$. 
Due to the one-to-many correspondence, enumerating $\Sol(f)$ by directly approximating the real roots of $w$ may be incorrect.
The real roots can be computed numerically; by evaluating the values of $v_1,v_2$ at these points, we obtain four approximate solutions of $\CP(f)$ as follows:
\begin{center}
    \hieu{\small \begin{tabular}{cc}
 (0.9999999996, 0.9999999996),  & (1.0000000000, 1.0000000000),  \\
(1.0000000000, 1.0000000000), & (1.0000000100, 1.0000000100),
    \end{tabular}}
\end{center}
which are different, but very close to the actual solution $(1,1)$.

From the above issue, in order to enumerate the points in $\Sol(f)$ numerically, we propose the following procedure using a given precision pair $\gamma=(\gamma_1,\gamma_2)$, where $\gamma_1> 0$ and  $\gamma_2> 0$, including two phases as follows:

\begin{itemize}
    \item \textit{Phase 1}. We first compute a univariate representation $[w,v_1,\dots,v_n]$ of $\Sol(f)$ by using \Cref{alg:Sol}. 
With the precision $\gamma_1$, we next compute numerically the real roots of $w$ and denote by 
$$A=\big\{\big(v_1(\tau),\dots,v_n(\tau)\big): \tau \text{ is a real root of } w \text{ with precision } \gamma_1\big\}.$$ 
\end{itemize}

For each actual solution $\alpha$, there can be $2^n$ points in $A$, denoted by the set $A_{\alpha}$, which are close to $\alpha$.
\begin{itemize}
    \item \textit{Phase 2}.  We then reduce $A_{\alpha}$ to one point with precision $\gamma_2$ using the following criterion: if $\alpha^1$ and $\alpha^2$ are in $A$ and $\|\alpha^1-\alpha^2\|<\gamma_2$, then we remove one point of them from the set $A$. 
\end{itemize}

Going back to \Cref{ex:CP1}, we take $\gamma_1=10^{-10}$ and $\gamma_2=10^{-6}$, after Phase 2, the four approximate solutions are reduced to $(1,1)$.

\begin{remark}
The method using semidefinite relaxations in \cite{zhao2019semidefinite} can be modified to enumerate the solutions of $\D_0$ problems, since they have finitely many solutions. Our method in this paper does this task without using semidefinite relaxations; this method uses the Shape Lemma employed in recent polynomial optimization work, e.g., \cite{magron2023,hieu2023computing}. A method using Gr\"{o}bner bases to compute solutions to such problems has been briefly described in \cite{yang1999grobner}, but the system presented there still contains inequalities. Furthermore, the state-of-the-art algorithms for computing real roots of polynomial systems having finitely many complex zeros by computer algebra methods are discussed in \cite{berthomieu}.
\end{remark}

\subsection{Experiments on Small-scale Problems with Worst-case Scenarios}

It should be noted that complementarity problems governed by nonlinear or tensor/polynomial data (without a special structure) are generally hard to solve because of the algorithmic difficulties that arise with higher degrees. Thus, the sizes solved in other papers, e.g. \cite[Section 4]{zhao2019semidefinite}, are small. The present paper also faces the same issue.

To see the computational behavior of the procedure proposed in \Cref{subsec:enumerate}, we report our experiments on small-scale problems with rational coefficients such that their numbers of solutions are maximal. 

The procedure is implemented in {\sc Maple}, and the results are obtained on an Intel i7 - 8665U CPU (1.90GHz) with 32 GB of RAM.
The package $\mathtt{Groebner}$ can handle symbolic tasks in \Cref{alg:Sol}. To compute the radical of an ideal, we use the command $\mathtt{Radical}$; to compute the reduced Gr\"{o}bner basis of an ideal with respect to a given lexicographic order, we use the command $\mathtt{Basis}$. 
To enumerate solutions, we use numerical computations by isolating real roots of $w$. We set the variable $\texttt{Digits}:= 10+\deg w$, where $\texttt{Digits}$ controls the number of digits that {\sc Maple} uses when making calculations with software floating-point numbers, and the second accuracy $\gamma_2$ is fixed to $10^{-4}$.

We perform our experiment with two families of small-scale complementarity problems. The first family is given by the polynomials in \Cref{rmk:degTCP2} with $n=2$, i.e. $p=(p_1,p_2)$, where $p_i=(x_i-1)(x_i-2)\dots (x_i-d)$ 
for $i=1,2$.
Let $d$ be in $\{2,4,6\}$, the computational results of the procedure described in \Cref{subsec:enumerate} are given in the following table:

\begin{center}
{\small
\begin{tabular}{c|cc|cc|cc|cc|}
\cline{2-9}
  & \multicolumn{2}{c|}{Alg 1} & \multicolumn{2}{c|}{all solutions} & \multicolumn{2}{c|}{ln solutions}      & \multicolumn{2}{c|}{sp solutions}      \\ \hline
\multicolumn{1}{|c|}{$d$} & $\deg w$       & time      & $\#\Sol$           & time          & \multicolumn{1}{c|}{$\#\Solln$} & time & \multicolumn{1}{c|}{$\#\Solsp$} & time \\ \hline
\multicolumn{1}{|c|}{2}   & 36             &     0.59      & 9                  & 9.40          &  \multicolumn{1}{c|}{1}       &   0.80   & \multicolumn{1}{c|}{1}           &  1.10     \\
\multicolumn{1}{|c|}{4}   & 100            &    6.95       & 25                 & 80.8         & \multicolumn{1}{c|}{1}       &  11.7    & \multicolumn{1}{c|}{1}           & 4.35     \\
\multicolumn{1}{|c|}{6}   & 144            &  119         & 49                 & 616          & \multicolumn{1}{c|}{1}        & 98.3     & \multicolumn{1}{c|}{1}           &  53.7    \\ \hline
\end{tabular}
}
\medskip

\textbf{Table 1}. Computational results on $p$
\end{center}

In Table 1, the time is given in seconds; $\#\Sol$, $\#\Solln$, and $\#\Solsp$ are respectively the cardinality of the solution set, least-norm solution set, and sparse solution set of $\CP(p)$.

The second family is quadratic complementarity problems, i.e. $d=2$ is fixed, given $q=(q_1,\dots,q_n)$, where $q_i=x_i^2-3x_i+1$ for $i\in[n]$. Let $n$ run from $1$ to $3$, the computational results are given in the following table:

\begin{center}
{\small
\begin{tabular}{c|cc|cc|cc|cc|}
\cline{2-9}
& \multicolumn{2}{c|}{Alg 1} & \multicolumn{2}{c|}{all solutions} & \multicolumn{2}{c|}{ln solutions}      & \multicolumn{2}{c|}{sp solutions}      \\ \hline
\multicolumn{1}{|c|}{$n$} & $\deg w$       & time      & $\#\Sol$           & time          & \multicolumn{1}{c|}{$\#\Solln$} & time & \multicolumn{1}{c|}{$\#\Solsp$} & time \\ \hline
\multicolumn{1}{|c|}{1}   & 6              &    0.12       & 3                  &     0.44          & \multicolumn{1}{c|}{1}           &    0.14  & \multicolumn{1}{c|}{1}           &   0.26   \\
\multicolumn{1}{|c|}{2}   & 36             &  0.65         & 9                  & 9.65              & \multicolumn{1}{c|}{1}           &  0.79   & \multicolumn{1}{c|}{1}           & 1.13     \\
\multicolumn{1}{|c|}{3}   & 216            &     126      & 27                 &    1084           & \multicolumn{1}{c|}{1}           & 120     & \multicolumn{1}{c|}{1}           &    4474  \\ \hline
\end{tabular}
}

\medskip
\textbf{Table 2}. Computational results on $q$
\end{center}

In Table 2, for the case $n=3$, the time to compute the sparse solution set is very expensive because of the computation of $w_{[3]}$ - the greatest common divisor of $w,v_1,v_2$, and $v_3$ whose degrees are 216 or 215 with large-scale coefficients; the time for only this task is $4431$ seconds.

\section{Solving Copositive PCPs}\label{sec:copositive}
This section is devoted to solving another class of PCPs that may not be $\D_0$. In particular, we focus on PCPs where involved mappings are copositive and approximate them by $\D_0$ problems. 

One says that a polynomial mapping $g:\R^n\to \R^n$ is \textit{copositive} if $\left\langle \alpha \mid  g(\alpha) \right\rangle \geq 0$ for all $\alpha$ in $\R_+^n$. Assume that $\deg g = d$. Denote by $g^{\infty}$ the ``leading" term of $g$, meaning that $g_i^{\infty}$ is the homogeneous term of degree $d$ of $g_i$, for $i\in[n]$.

Recall from \cite[Corollary 7.2]{gowda2016polynomial} that if $f$ or $f^{\infty}$ is copositive and $\Sol(f^{\infty})=\{0\}$, then $\Sol(f+a)$ is nonempty and compact for every vector $a$ in $\R^n$.

\begin{proposition}\label{prop:genericity2} Suppose that $f$ or $f^{\infty}$ is copositive and $\Sol(f^{\infty})=\{0\}$. The following statements hold:

\begin{itemize}
    \item[$\rm a)$] There exists a generic set $\mathcal{O}$ in $\R^n$ such that $\CP(f+a)$ is $\D_0$ for every $a\in \mathcal{O}$.
    \item[$\rm b)$] For every $\mu>0$, there exists $\varepsilon > 0$ such that $\Sol(f+a)$ is nonempty and 
$\Sol(f+a)\subseteq \Sol(f) + \mathbb{B}(0,\mu)$
    for any $a\in \R^n$ satisfying  $\|a\| < \varepsilon$.
\end{itemize}
 \end{proposition}

 \begin{proof} The proof of the first statement is analogous to the proof of \Cref{prop:genericity}. 
 The second one is obtained by applying directly \cite[Corollary 7.2]{gowda2016polynomial} and \cite[Corollary 4.1]{hieu2020solution} to polynomial complementarity problems.
 \qed
 \end{proof}

It is noteworthy that \Cref{prop:genericity2} enables us to solve approximately a PCP satisfying the assumptions therein: We take $a\in\R^n$ randomly \hieu{(then $f+a$ is $\D_0$)} and $\|a\|$ small enough, and then solve $\CP(f+a)$ by using the method proposed in \Cref{subsec:enumerate}. A solution of $\CP(f+a)$ can be considered as an approximate solution to the original problem.

\begin{example} Consider the polynomials in two variables $x_1,x_2$ given in \cite[Example 5.3]{hieu2019r0}, $f_1=f_2=x_1^2+x_2^2-1$. It is easy to check that $f_1^{\infty}=f_2^{\infty}=x_1^2+x_2^2$ and $f^{\infty}$ is copositive and that $\Sol(f^{\infty})=\{0\}$. Moreover, $\CP(f)$ has infinitely many solutions
$$ \Sol(f)= \{(\alpha_1,\alpha_2)\in\R^2: \alpha_1^2 + \alpha_2^2=1, \alpha_1 \geq 0,\alpha_2\geq 0\}.$$
We observer that $\CP(f+ a)$, where $a=(a_1,a_2)$, is $\D_0$ if any only if $a_1\neq a_2$. Taking $a_1=1/10^6$ and $a_1=1/10^5$ and computing solutions to $\CP(f+a)$ by using the proposed method, the perturbed problem has unique solution $\alpha_a=(0.999999, 0.000000)$ which is very close to the solution $(1,0)$ of $\CP(f)$.
    
\end{example}

\section*{Conclusions}
We introduced a class of generic polynomial complementarity problems, named $\D_0$, and proposed algorithms to compute univariate representations of their solutions.  We also presented a method for solving copositive problems by approximating them with  $\D_0$ problems. The experimental results showed that these algorithms and method are efficient for small-scale problems.

\appendix

\section*{Appendix: Proof of Theorem \ref{thm:Sard_para}}
The proof of Theorem \ref{thm:Sard_para} is analogous to the proof of \cite[Theorem 1.10]{pham2016genericity}, which is presented in the semi-algebraic setting. Therefore, only the most important modifications are shown below.

\begin{proof} We only need to prove that there exists a set $\mathcal{O}$ in $\R^m$ such that the three following conditions hold: $\R^m\setminus \mathcal{O}$ has Lebesgue measure zero; $\mathcal{O}$ is semi-algebraic; and,  for any $c \in \mathcal{O}$, $0$ is a regular value of the mapping $\Phi_{c}$.

We assume without loss of generality that $\Phi^{-1}(0)$ is non-empty. It follows from the Implicit Function Theorem that $\Phi^{-1}(0)$ is a smooth sub-manifold in $\R^m\times V$, where $\Phi^{-1}(0)$ is considered as a smooth manifold in $\R^m\times \R^{2n}$.

Consider the restriction of $\pi$, where $\pi: \R^m\times V \to \R^m, (c,x)\mapsto c$ is a natural projection, on $\Phi^{-1}(0)$.  By Sard's theorem \cite[Theorem 6.10]{lee2003smooth} (which applies to smooth mappings between real manifolds), the set of critical
values of the restriction mapping $\pi|_{\Phi^{-1}(0)}$, denoted by $\mathcal{C}$, has Lebesgue measure zero in $\R^m$. 
It is not hard to see that $\mathcal{C}$ is a semi-algebraic set in $\R^m$. We define $\mathcal{O}:=\R^m\setminus\mathcal{C}$.
Consequently, $\mathcal{O}$ is a semi-algebraic set and $\R^m\setminus\mathcal{O}$ has Lebesgue measure zero  in $\R^m$.

By repeating the argument in the proof of \cite[Theorem 1.10]{pham2016genericity}, we can show that $0$ is a regular value of the mapping $\Phi_{c} $ for any $c \in \mathcal{O}$. 
\qed
\end{proof}

\begin{acknowledgements}
The authors would like to thank the editor and anonymous referees for their comments and Jo\~{a}o M. Pereira for his feedback on \Cref{rmk:degTCP2}.  
\end{acknowledgements}

\noindent \small{\textbf{Funding} This work was supported by the JSPS KAKENHI Grant Number 23H03351.}
\bibliographystyle{spmpsci.bst}
\bibliography{references.bib}

\end{document}